\documentclass{amsart}

\usepackage{verbatim}

\DeclareMathOperator\co{{\it\mathbb C}}
\DeclareMathOperator{\ba}{\mathbb B}

\newtheorem{theorem}{Theorem}[section]
\newtheorem{lemma}[theorem]{Lemma}
\newtheorem{prop}[theorem]{Proposition}
\newtheorem{corollary}[theorem]{Corollary}

\theoremstyle{definition}
\newtheorem{definition}[theorem]{Definition}

\theoremstyle{remark}
\newtheorem{remark}[theorem]{Remark}

\numberwithin{equation}{section}

\begin{document}

\title[Complex Monge-Amp\`ere equation on Hermitian manifolds]{Weak solutions to the complex Monge-Amp\`ere  equation on Hermitian manifolds }


\author{S\l awomir Ko\l odziej}
\address{Faculty of Mathematics and Computer Science, Jagiellonian University 30-348 Krak\'ow, \L ojasiewicza 6, Poland.}
\email{Slawomir.Kolodziej@im.uj.edu.pl}

\author{Ngoc Cuong Nguyen}
\address{Faculty of Mathematics and Computer Science, Jagiellonian University 30-348 Krak\'ow, \L ojasiewicza 6, Poland.}
\email{Nguyen.Ngoc.Cuong@im.uj.edu.pl}





\maketitle

\quad \quad {\em Dedicated to Duong H. Phong  on the occasion of his 60th birthday}

\bigskip

\begin{abstract}
The main result asserts the existence of continuous solutions of the complex Monge-Amp\`ere  equation
with the right hand side in $L^p , p>1$, on compact Hermitian manifolds.

\end{abstract}

\section*{Introduction}

Let $(X,\omega)$ be a compact Hermitian manifold of complex dimension $n$. We study the 
weak solutions to the complex Monge-Amp\`ere equation
\[
	(\omega + dd^c \varphi)^n = f \omega^n,	\quad
	\omega + dd^c \varphi \geq 0,
\]
where $0\leq f \in L^p(X, \omega^n)$, $p>1$, and $d^c= \frac{i}{2\pi} (\bar\partial - \partial)$, $dd^c=\frac{i}{\pi}\partial \bar\partial$, with the displayed inequality understood in the
 sense of currents.

We follow the pluripotential approach introduced by 
S. Dinew and the first author in \cite{DK}, where $L^{\infty }$ estimates for the above equation
were obtained. Here we refine those estimates and prove the existence of continuous solutions.

\begin{theorem}
Let $(X, \omega)$ be a compact Hermitian manifold, $dim \,  X =n$.
Let $0 \leq f \in L^p(X, \omega^n)$, $p>1$, be such that $\int_X f \omega^n >0$.
There exist a constant $c>0$ and a function $u\in C(X)$ satisfying the
equation
\[
	(\omega + dd^c u)^n = c f \, \omega^n, \quad \omega + dd^c u \geq 0,
\]
in the weak sense.
\end{theorem}

The main tool is a generalized version of
the comparison principle due to Bedford-Taylor \cite{BT1, BT2}. We call it \it modified comparison principle \rm
just for a convenient reference.
In its formulation we use a constant $B > 0$ such that
\begin{equation}
\label{curvature}
\begin{cases}
	-B \omega^2 \leq 2 n dd^c \omega \leq B \omega^2, \\
	-B \omega^3 \leq 4 n^2 d\omega \wedge d^c \omega \leq B \omega^3. 
\end{cases} 
\end{equation}
We denote by $PSH(\omega)$  the set of $\omega$-plurisubharmonic functions on $X$ (see Section~\ref{preliminary}).
 
\begin{theorem}[modified comparison principle]
Let $(X,\omega)$ be a compact Hermitian manifold and
suppose that $\varphi, \psi \in PSH(\omega)\cap L^\infty(X)$. Fix $0 < \varepsilon < 1$ and set
$m(\varepsilon) = \inf_X [ \varphi - (1-\varepsilon) \psi]$. 
Then, for any $0 < s < \frac{\varepsilon^3}{16 B}$, we have 
\[
	\int_{ \{ \varphi < (1-\varepsilon) \psi + m(\varepsilon) + s\}}
		\omega_{(1-\varepsilon) \psi}^n
	\leq (1 +  \frac{s}{\varepsilon^n} \, C)
	 \int_{ \{ \varphi < (1-\varepsilon) \psi + m(\varepsilon) + s\}}
		\omega_{\varphi}^n,
\]
where $C$ is a uniform constant depending only on $n, B$.
\end{theorem}

It was shown in \cite{DK} that the comparison principle which is valid on K\"ahler manifolds (see \cite{ko03})
is no longer true on general Hermitian manifolds.

The complex Monge-Amp\`ere equation on complex Hermitian manifolds was first
studied by Cherrier \cite{Ch, CH1, CH2} and  Hanani  \cite{H1, H2}.
 There has been a renewed interest recently in the works of Guan - Li \cite{GL} and  Tosatti - Weinkove \cite{TW1, TW2}.
The breakthrough was made by Tosatti and Weinkove \cite{TW2} who proved the existence and uniqueness of the smooth 
solution to  complex Monge-Amp\`ere equation on a general compact Hermitian manifold. 
Since then more papers appeared (e.g., \cite{Gil1}, \cite{GS}, \cite{Nie}, \cite{Sun}, \cite{ZZ}), some
 in relation to  the Chern-Ricci flow.
It was shown in \cite{Gil1, Gil2}, \cite{TW3, TW4, TWYang} that the flow enjoys many
 common properties with the K\"ahler-Ricci flow. 
In the study of the latter the weak solutions 
of the complex Monge-Amp\`ere equation play an important role, and thus the investigation of
the Hermitian case seems to be well motivated.

The method based on the modified comparison principle can also be applied 
in the case when $X= \Omega$  is a bounded open set in $\co^n$. We consider the Dirichlet 
problem for the Monge-Amp\`ere operator 
and generalize the stability
estimates \cite{ko96}  from the K\"ahler setting  to the Hermitian one.
\begin{corollary}
Consider $\Omega$ be a bounded open set in $\co^n$ and  $\omega$ be a Hermitian metric
in $\co^n$. Let  $u,v \in PSH( \omega) \cap C(\bar \Omega)$ be such that
 \[
 (\omega + dd^c u)^n =f \omega^n, \quad  
(\omega + dd^c v)^n = g \omega^n,
\]
with $0 \leq f, g \in L^p(\Omega, \omega^n)$, $p>1$.
Then 
\[
	\| u - v \|_{L^\infty(\bar\Omega)} 
			\leq \sup_{\partial \Omega} |u-v| 
				+ C \| f-g\|_{L^p(\omega^n)}^\frac{1}{n},
\]
where C depends only on $\Omega$, $\omega$ and $p$.
\end{corollary}

Thanks to the domination principle and the stability estimate the Dirichlet problem for
Monge-Amp\`ere operator (with the background metric $\omega$) is solvable for the right hand side in $L^p$, $p>1$.

\begin{corollary}
There exists a unique continuous solution to the Dirichlet
problem \eqref{dirichlet-problem} in a $C^\infty$ strictly pseudoconvex domain.
\end{corollary}

The note is organized as follows. We recall some basic properties of $\omega$-
plurisubharmonic functions on complex Hermitian manifolds in {\em Section~\ref{preliminary}}.
{\em Section~\ref{global-cp}} is devoted
to prove the modified comparison principle. Then the domination principle in the local case is inferred in {\em Section~\ref{dom-prin}}. 
The  stability estimates and the Dirichlet problem for complex Monge-Amp\`ere in 
a bounded domain in $\co^n$ are studied in {\em Section~\ref{local-dp}}. In {\em Section~\ref{global-dp}} we show  $L^\infty$
{\em a priori} estimates and the existence of continuous solutions to complex
Monge-Amp\`ere equations on a compact Hermitian manifold. 

\medskip

\bf Dedication  \rm 
It is a great honour for the authors to dedicate this paper to Duong H. Phong in appreciation of his wisdom which reaches far beyond mathematics.

\medskip

\bf Acknowledgements  \rm The first author was partially supported by NCN grant 2011/01/B/ST1/00879. The second author is supported 
by the International Ph.D Program {\em" Geometry and Topology in Physical Models "}. He also would like to thank Szymon Pli\'s for helpful discussions.

\section{Basic properties of $\omega$-psh functions in the  Hermitian setting}
\label{preliminary}

Let $\Omega$ be an open set in $\co^n$
and $\omega$ a Hermitian metric in $\co^n$. 
We  collect here some basic facts about $\omega$-plurisubharmonic 
($\omega$-psh for short) functions.
We refer to \cite{DK} for more discussion. 
Recall that we use the normalisation 
$d = \partial + \bar\partial, d^c= \frac{i}{2\pi}(\bar\partial - \partial)$, $dd^c = \frac{i}{\pi} \partial\bar\partial$.
\begin{definition}
\label{psh}
Let $u : \Omega \rightarrow [-\infty, +\infty [$ be a upper semi-continuous. 
Then $u$ is called $\omega$-psh if $u\in L^1_{loc} (\Omega, \omega^n)$ and 
$dd^c u + \omega \geq 0$ as a current.
\end{definition}
Denote by $PSH ( \Omega, \omega)$ the set of 
$\omega$-psh functions in $\Omega$
(when $\Omega$ is clear from the context, we write $PSH(\omega)$ ).
We often  use the short-hand notation $\omega_u := (\omega+ dd^c u)$. 
Following Bedford-Taylor \cite{BT2}, one  defines the wedge product 
\[
	\omega_{v_1}\wedge ...\wedge \omega_{v_k}
\]
for $v_1, ..., v_k \in PSH(\omega) \cap L^\infty(\Omega)$, $1\leq k \leq n$; proceeding  by induction over $k$. 
For $k=1$ the definition is given by classical distribution theory. Suppose that
for $1\leq k \leq n-1$ the current
\[
	T = \omega_{v_1} \wedge ...\wedge \omega_{v_k} 
\]
is well defined. Fix a small ball $\ba$ in $\Omega$ 
and  a strictly psh function $\rho$  such that $dd^c \rho \geq 2 \omega$ in $\ba$.
Put $\gamma = dd^c \rho - \omega$ and $u_l = \rho + v_l \in PSH(\ba) \cap L^\infty(\ba)$, 
then $T$ can be written in $\ba$ as a linear 
combination of positive currents 
\begin{equation}
\label{combination}
	dd^c u_{j_1} \wedge ...\wedge dd^c u_{j_l} \wedge \gamma^{k-l},
	\quad 1 \leq j_1 < ...< j_l \leq k, \; 1\leq l \leq k.
\end{equation}
We know that there are sequences of smooth $\omega$-psh function $\{v_l^j\}_{j=1}^\infty$ which 
decrease to $v_l$, $1\leq l \leq k$
(by  Demailly's regularization theorem for quasi-psh functions). 
Since $T$ is a linear combination of positive currents of the form \eqref{combination},  
we obtain by the results from \cite{BT2}
\[
	T = \lim_{j \rightarrow \infty} T_j =
	\lim_{j\rightarrow \infty} \omega_{v_1^j} \wedge ...\wedge \omega_{v_k^j}
	\quad \mbox{ weakly. }
\]
Thus, $T$ is a positive current of bidgree $(k, k)$.
Moreover, 
\[
	dT = \sum_{l=1}^k d \omega \wedge 
		\omega_{v_1} \wedge ... \widehat{\omega}_{v_l}... \wedge \omega_{v_k};
\]
\[	d^c T = \sum_{l=1}^k d^c \omega \wedge 
		\omega_{v_1} \wedge ... \widehat{\omega}_{v_l}... \wedge \omega_{v_k};
\]
\[	dd^c T = 2 \sum_{1\leq l< m \leq k}  d \omega \wedge d^c \omega \wedge
		\omega_{v_1} \wedge ... 
		\widehat{\omega}_{v_l}...  \widehat{\omega}_{v_m}...
				\wedge \omega_{v_k} 
				+  \sum_{l=1}^k dd^c \omega \wedge \omega_{v_1} \wedge ...
					\widehat{\omega}_{v_l} ...\wedge \omega_{v_n}.
\]
The notation $\widehat{\omega}_{v_l}$ means that this term does not appear in 
the wedge product. Now we define for 
$u\in PSH(\omega) \cap L^\infty(\Omega)$
\[
	dd^c u \wedge T 
	:= dd^c (u \wedge T) - du \wedge d^c T + d^c u \wedge d T - u dd^c T.	
\]
The right hand side is well defined by the above formulas for $dT, d^c T$ and $dd^cT$.
Let $\{u^j\}_{j=1}^\infty$ 
be a sequence of smooth $\omega$-psh functions decreasing to $u$. We have
\[
	dd^c u\wedge T = \lim_{j \rightarrow \infty} dd^c u^j \wedge T_j
	\quad \mbox{ weakly}.
\]
Note here that for any test form $\varphi$ of bidgree $(n-k-1, n-k-1)$ 
\[
	d u \wedge d^c T \wedge \varphi= - d^c u \wedge dT \wedge \varphi.
\]
Thus,
\[
	\omega_u \wedge T = \omega \wedge T + dd^c u \wedge T
	:= \omega \wedge T + dd^c( u T) - 2 du \wedge d^c T - u dd^c T
\]
is a positive current of bidgree $(k+1, k+1)$. In the special case when 
$v_1 = ...= v_n =v \in PSH(\omega) \cap L^{\infty}(\Omega)$ we get the definition of
Monge-Amp\`ere operator 
\[
	\omega_v^n := \omega_v \wedge ...\wedge \omega_v,
\]
($n$ factors on the right hand side) which is a Radon measure.
Finally, we state for the later reference a convergence result which follows also from the corresponding statement
in \cite{BT2} applied to currents of the form \eqref{combination}.
\begin{prop}
\label{convergence}
Let $v_1, ..., v_k \in PSH(\omega) \cap L^\infty(\Omega)$, $1 \leq k \leq n$. 
Suppose that the sequences of bounded 
$\omega$-psh functions $\{v_1^j\}_{j=1}^\infty$, 
...,$\{v_k^j\}_{j=1}^\infty$ decrease (or uniformly converge) to $v_1, ..., v_k$ respectively.
Then
\[
	\lim_{j\rightarrow \infty} \omega_{v_1^j} \wedge ... \wedge \omega_{v_k^j}
	= \omega_{v_1} \wedge ...\wedge \omega_{v_k} \quad \mbox{weakly}.
\]
In particular, if $\{u_j\}_{j=1}^\infty \in PSH(\omega) \cap L^\infty(\Omega)$ decreases
(or uniformly converges) to $u \in PSH(\omega) \cap L^\infty(\Omega)$, then
\[
	\lim_{j \rightarrow \infty} \omega_{u_j}^n = \omega_u^n 
	\quad \mbox{weakly}.
\]
\end{prop}

Let now  $(X, \omega)$ be a compact Hermitian manifold, with $dim_{\co} X =n$.
The above (local)  construction applies in this  setting. 

\begin{definition}
\label{pshcpt}
Let $u : X \rightarrow [-\infty, +\infty [$ be an upper semi-continuous function. 
Then, $u$ is called  $\omega$-psh  if $u\in L^1 (X,\omega^n)$ and 
$dd^c u + \omega \geq 0$ as a current.
\end{definition}
Denote by $PSH ( \omega)$ the set of 
$\omega$-psh functions on $X$.
By the definition
$u \in PSH(\omega)$ if and only if $u \in PSH(\Omega, \omega)$ for any 
coordinate chart $\Omega \subset\subset X$.
Using partition of unity, we define the Monge-Amp\`ere operators 
$\omega_u^n$ for $u \in PSH(\omega) \cap L^\infty(X)$. It is also clear that 
Proposition~\ref{convergence} holds in this setting.

\section{The modified comparison principle}
\label{global-cp}

Let $(X,\omega)$ be a compact Hermitian manifold, $dim_{\co} X =n$. 
It is known  (see \cite{DK}) that
the comparison principle is not true on  a general compact Hermitian manifold. 
We shall use two lemmata to prove the main theorem of this section (Theorem~\ref{modified-comparison-principle}). 
From the proof  of Proposition 3.1 in \cite{BT1} and the approximation result in
\cite{BK} we have the following statement.

\begin{lemma}
\label{weak-cp1}
For $T:= (\omega+ dd^c v_1) \wedge ... \wedge (\omega + dd^c v_{n-1})$,  where 
$ v_1, ..., v_{n-1} \in PSH(\omega) \cap L^\infty(X)$  and for  $\varphi, \psi \in PSH(\omega)\cap L^\infty(X)$ we have
\[
	\int_{\{ \varphi < \psi \}} dd^c \psi \wedge T
	\leq \int_{\{ \varphi < \psi \}} dd^c \varphi \wedge T
	+ \int_{\{ \varphi < \psi \}} (\psi - \varphi) \, dd^c T.
\]
\end{lemma}

 A weaker version of the comparison principle was shown in  \cite{DK}.
\begin{lemma}
\label{weak-cp2}
Let $\varphi, \psi \in PSH(\omega) \cap L^\infty(X)$. 
Then there is a constant
$C_n = C(n)$ such that, for $B \sup_{\{\varphi < \psi \}} (\psi - \varphi ) \leq 1$,
\[
	\int_{\{ \varphi < \psi \}} (\omega + dd^c \psi)^n 
	\leq \int_{\{ \varphi < \psi \}} (\omega + dd^c \varphi)^n
	+ C_nB \sup_{\{\varphi < \psi \}} ( \psi - \varphi ) \sum_{k =0}^{n-1} 
		\int_{ \{\varphi < \psi \} } \omega_\varphi^k \wedge \omega^{n-k}.
\]
\end{lemma}

We are ready to prove the modified comparison principle.

\begin{theorem}
\label{modified-comparison-principle}
Let $\varphi, \psi \in PSH(\omega)\cap L^\infty(X)$. Fix $0 < \varepsilon < 1$ and set
$m(\varepsilon) = \inf_X [ \varphi - (1-\varepsilon) \psi]$. 
Then for any $0 < s < \frac{\varepsilon^3}{16 B}$,
\[
	\int_{ \{ \varphi < (1-\varepsilon) \psi + m(\varepsilon) + s\}}
		\omega_{(1-\varepsilon) \psi}^n
	\leq (1 +  \frac{s B}{\varepsilon^n} \, C)
	 \int_{ \{ \varphi < (1-\varepsilon) \psi + m(\varepsilon) + s\}}
		\omega_{\varphi}^n,
\]
where $C$ is a uniform constant depending only on $n$.
\end{theorem}

\begin{proof}
We wish to apply Lemma~\ref{weak-cp2} with 
 $ (1-\varepsilon) \psi + m(\varepsilon) + s $ in place of $\psi$.
 Note   that on $U(\varepsilon, s ) = \{ \varphi < (1-\varepsilon) \psi + m(\varepsilon) + s\}$,
\[
	\sup_{U(\varepsilon, s )} \left[ (1-\varepsilon) \psi + m(\varepsilon)  - \varphi +s
							\right]
					\leq s.
\]
Therefore, in view of Lemma~\ref{weak-cp2},  it is enough to estimate
\[
	 \sum_{k =0}^{n-1} 
		\int_{ U(\varepsilon, s) } \omega_\varphi^k \wedge \omega^{n-k}.
\]
For $k = 0, ..., n$, set 
\[ 
	a_k = \int_{U(\varepsilon,s)} \omega_{\varphi}^k\wedge \omega^{n-k}.
\]
Let $\delta := \frac{\varepsilon^3}{16 B}$. 
We  shall verify that for $0< s < \delta$ 
\begin{equation}
\label{mixma1}
	\varepsilon  \, a_0 \leq a _1 +  \delta \, B\,  a_0, \quad \mbox{ and } \quad
	\varepsilon \, a_1 \leq a_2 + \delta\, B \, (a_1 + a_0),
\end{equation}
and for $2 \leq k \leq n-1$,
\begin{equation}
\label{mixma2}
	\varepsilon \, a_k \leq a_{k+1} + \delta \,B \, ( a_k + a_{k-1} + a_{k-2}).
\end{equation}
Let us assume for a moment that \eqref{mixma1} and \eqref{mixma2} are true.
 It follows from the first inequality of \eqref{mixma1} that
\begin{equation}
\label{mixma3}
	a_0 \leq d_1  \, a_1 \; \mbox { with } d_1= \frac{1}{ \varepsilon -  \delta \, B  }.
\end{equation}
From the second inequality of \eqref{mixma1} and \eqref{mixma3} we have
\[
	a_0 \leq d_1 \, d_2 \, a_2 \quad \mbox{ and } \quad a_1 \leq d_2 \, a_2,
\]
with $1/{d_2} := \varepsilon - \delta \, B\, (1 + d_1)$. Using \eqref{mixma2} and the induction we get that, for $k =0 ,..., n-1$,
\begin{equation}
\label{mixma4}
	a_k \leq d_{k+1} \, ... \, d_{n} \, a_n
\end{equation}
where  $d_0 :=0$, $1/d_1= \varepsilon - \delta \,B$, and for $j \geq 1$, 
\[
	1/d_{j+1} = \varepsilon - \delta\, B \, (1 + d_j + \, d_{j-1}\, d_j).
\]
Furthermore, since $\delta \, B= \frac{\varepsilon^3}{16}$,
by an elementary calculation, one gets that
\begin{equation}
\label{p}
	\varepsilon^{-1} < d_j <2\varepsilon^{-1} \quad \forall j \geq 1.
\end{equation}
In particular $d_j$ are positive and finite.
It concludes for any $0 \leq k \leq n-1$  and for $0< s < \delta$,
\[
	a_k \leq d_{k+1}... d_n \,  a_n \leq \frac{C}{\varepsilon^n} \, a_n.
\]
It remains to verify \eqref{mixma2} (as
\eqref{mixma1} is its consequence with the convention that $a_k = 0$ for $k< 0$).
Indeed, since 
\[ 
	\varepsilon \, \omega \leq
	\omega + dd^c [(1-\varepsilon) \psi + m(\varepsilon) +s ] \quad \mbox{ and } \quad
	U(\varepsilon, s ) = \{ \varphi < (1-\varepsilon) \psi + m(\varepsilon) + s\},
\]
it follows from Lemma~\ref{weak-cp1} that
\[
	\varepsilon \int_{U(\varepsilon, s)} \omega_\varphi^k \wedge \omega^{n-k}
\leq	\int_{U(\varepsilon, s)} \omega_{(1-\varepsilon) \psi} 
					\wedge \omega_\varphi^k \wedge \omega^{n-k-1}
\leq	\int_{U(\varepsilon, s )} \omega_\varphi^{k+1} \wedge \omega^{n-k-1}
	+ R, 
\]
where 
\[
	R	= \int_{U(\varepsilon, s)} 
			[(1-\varepsilon) \psi + m(\varepsilon) + s - \varphi] 
				dd^c \left( \omega_\varphi^k \wedge \omega^{n-k-1}
					\right)
		\leq s\, B\, ( a_k + a_{k-1} + a_{k-2}).
\]
Thus, for $0 < s < \delta = \frac{\varepsilon^3}{16 B}$,
\[
	\varepsilon a_k \leq a_{k+1} + \delta \, B(a_k + a_{k-1} + a_{k-2}).
\]
The theorem follows.
\end{proof}

\section{The domination principle}
\label{dom-prin}

Let $\Omega$ be a bounded open set in $\co^n$. The constant $B>0$ is defined as in \eqref{curvature} for $\bar \Omega$. The next theorem is
an analogue of the modified comparison principle for a bounded open set
in $\co^n$.

\begin{theorem}
\label{local-modified-comparison-principle}
Fix $0 < \varepsilon < 1$. Let $\varphi, \psi \in PSH(\omega)\cap L^\infty(\Omega)$ 
be such that $\liminf_{\zeta \rightarrow z\in \partial \Omega} (\varphi - \psi)(\zeta) \geq 0$.  
Suppose that $M = \sup_\Omega (\psi - \varphi)>0$, and  
$\omega + dd^c \psi \geq \varepsilon \omega$ in $\Omega$. Then, for any 
$0 < s < \varepsilon_0:= \min\{ \frac{\varepsilon^n}{16 B}, M \}$, 
\[
	\int_{ \{ \varphi < \psi - M  + s\}}
		\omega_{\psi}^n
	\leq \left(1 +  \frac{sB}{\varepsilon^n} \, C_n \right)
	 \int_{ \{ \varphi < \psi -M + s\}}
		\omega_{\varphi}^n,
\]
where $C_n$ is a uniform constant depending only on $n$.
\end{theorem}

\begin{proof}
It is very similar to the proof of the modified comparison principle. The
 lemmata we need have now the following form.
\begin{lemma}
\label{local-weak1}
Let $T:= (\omega+ dd^c v_1) \wedge ... \wedge (\omega + dd^c v_{n-1})$ with 
$ v_1, ..., v_{n-1} \in PSH(\omega) \cap L^\infty( \Omega)$  be a positive current of bidegree $(n-1, n-1)$. 
Let $\varphi, \psi \in PSH(\omega) \cap L^\infty(\Omega)$. 
If $\liminf_{ \zeta \rightarrow z \in \partial\Omega} (  \varphi - \psi) (\zeta) \geq 0$, then
\[
	\int_{\{ \varphi < \psi \}} dd^c \psi \wedge T
	\leq \int_{\{ \varphi < \psi \}} dd^c \varphi \wedge T
	+ \int_{\{ \varphi < \psi \}} (\psi - \varphi) \, dd^c T.
\]
\end{lemma}

\begin{lemma}
\label{local-weak2}
Let $\varphi, \psi \in PSH( \omega)\cap L^\infty(\Omega)$ be such that 
$\liminf_{\zeta \rightarrow z\in \partial \Omega} (\varphi - \psi)(\zeta) \geq 0$. 
Suppose that $B \sup_{\{\varphi < \psi \}} ( \psi - \varphi ) \leq 1$. Then,
\begin{align*}
	\int_{\{ \varphi < \psi \}} \omega_\psi^n 
	\leq \int_{\{ \varphi < \psi \}} \omega_\varphi^n 
 	+ B \sup_{\{\varphi < \psi \}} ( \psi - \varphi ) \left(  C_n \sum_{k =0}^{n-1} 
		\int_{ \{\varphi < \psi \} } \omega_\varphi^k \wedge \omega^{n-k} \right),
\end{align*}
where the constant $C_n$ depends only on $n$.
\end{lemma}
 Having those the proof goes exactly as  the one of Theorem~\ref{modified-comparison-principle}. 
\end{proof}

As a consequence we obtain the domination principle.

\begin{corollary}
\label{domination-principle}
Let $\Omega$ be a bounded open set in $\co^n$. 
Let $u, v \in PSH(\omega) \cap L^\infty( \Omega)$ be such that $\liminf_{ \zeta 
\rightarrow z \in \partial\Omega} (  u-v) (\zeta) \geq 0$. 
Suppose that $(\omega +dd^c u )^n \leq (\omega + dd^c v)^n$. Then $v \leq u$ in $\Omega$.
\end{corollary}

\begin{proof}
First, we may assume that 
$\liminf_{ \zeta \rightarrow z \in \partial\Omega} (  u -v ) (\zeta) \geq 2\alpha >  0$.
Otherwise, replace $u$ by $u + 2 \alpha$ and then let $\alpha \rightarrow 0$.
Thus there is a relatively compact open set $\Omega'$ such that 
$u(z) \geq v(z)+ \alpha$ for $z\in \Omega \setminus \Omega'$. 
By subtracting the same constant, we also assume that $u, v \leq 0$.
We argue by contradiction. Suppose that $\{ u < v \}$ is non empty.
Since $\Omega$ is bounded, 
there is a strictly psh function $\rho \in C^2(\bar\Omega)$ 
such that $-C \leq \rho \leq 0$ in $\Omega$,  for some constant $0< C$. Since,
$u, v ,\rho$ are bounded in $\Omega$, then after multiplying $\rho$ by 
a small positive constant
we see that there exist $0< \varepsilon, \tau <<1/2$  such that 
\[
	dd^c \rho \geq 2\,  \varepsilon\, \omega , \quad 
	(1 - \tau)^{1/n} + (2\tau)^{1/n} \leq 1+  \varepsilon,
\]
and
\[	
	\{ u < (1-\tau)^{1/n} v + (2\tau)^{1/n} u + \rho\}\subset\subset \Omega
\]
is non empty. Put $\hat v := (1-\tau)^{1/n} v + (2\tau)^{1/n} u + \rho$. 
Since $\omega_v^n \geq \omega_u^n$, it follows that
\[
	\omega_{\hat v}^n 
				\geq \left[ (1-\tau)^{1/n} \omega_v + (2\tau)^{1/n} \omega_u
						\right]^n
				 \geq (1-\tau) \, \omega_v^n + 2 \tau \, \omega_u^n
				 \geq (1+\tau) \omega_u^n.
\]
Thus,
\begin{equation}
\label{major}
	\omega + dd^c \hat v \geq \varepsilon \omega 
		\quad	\mbox{ and } \quad	
	\omega_{\hat v}^n \geq (1+ \tau)\, \omega_u^n 
\end{equation}
in   $\Omega $.
Let us denote  by $U(s)$  the set $ \{ u < \hat v -M +s \}$ with $M = \sup_\Omega (\hat v - u) >0$. 
Then for any $0< s < M$, 
\[
	U(s) \subset\subset \Omega
	\quad \mbox{ and } \quad
	\sup_{U(s)} \{ (\hat v -M+s) -u \}=s.
\]
It follows from \eqref{major} that the assumptions of Theorem~\ref{local-modified-comparison-principle} are fulfilled for 
$\varphi:= u$, $\psi := \hat v -M +s$. 
Hence, for any $0 < s < \epsilon_0 = \min\{ \frac{\varepsilon^n}{16 B}, M \}$,
\[
	0< \int_{U(s)} ( \omega + dd^c \hat v )^n 
	\leq   \left(1 +  \frac{sB}{\varepsilon^n} \, C_n \right)
		 \int_{ U(s)}
			\omega_u^n.
\]
Then using \eqref{major}, we get for $ 0 < s < \epsilon_0$  
\begin{equation}
\label{ci1}
	0< \tau \int_{U(s)} \omega_u^n 
	\leq    \frac{ s \, B C_n}{\varepsilon^n} 
		 \int_{ U(s)}
			\omega_u^n.
\end{equation}
Therefore $0< \tau \leq  \frac{ s \, B C_n}{\varepsilon^n}$.
This is impossible when $0< s$ is small enough. Thus, the proof follows.
\end{proof}

\section{The Dirichlet problem in a bounded domain in $\co^n$}
\label{local-dp}

Denote by $\beta$ the standard K\"ahler form $ dd^c \|z\|^2$  in $\co^n$ and 
by  $\omega$  an arbitrary  Hermitian form in $\co^n$.
 Let $\Omega$ be a bounded open
set in $\co^n$. We write $L^p(\omega^n)$ for $L^p(\Omega, \omega^n)$ 
 and  consider the Dirichlet problem for the Monge-Amp\`ere equation
with  the background metric  $\omega$.
Given $0 \leq f \in L^p(\omega^n)$, $p>1$, and $\phi \in C(\partial\Omega)$, 
we seek for a solution to
\begin{equation}
\label{dirichlet-problem}
\begin{cases}
	u \in PSH(\omega)\cap C(\bar \Omega), \\
	(\omega+ dd^c u)^n = f(z)\;\omega^n & \mbox { in }
		\Omega, \\ 
	u  = \phi &\mbox{ on } \partial \Omega, 
\end{cases}
\end{equation}
where the  equality in the second line is understood in  sense of currents.

From the domination principle above  and the stability estimates  \cite{ko96} we get the following result.
\begin{theorem}
\label{stability-local} 
Let $\Omega$ be a bounded open set in $\co^n$ and
let  $u,v \in PSH( \omega) \cap C(\bar \Omega)$ be such that
\[
	 (\omega + dd^c u)^n =f \omega^n, \quad
	(\omega + dd^c v)^n = g \omega^n
\] 
with $0\leq f, g \in L^p(\omega^n)$, $p>1$.
Then 
\[
	\| u - v \|_{L^\infty(\bar \Omega)} 
			\leq \sup_{\partial \Omega} |u-v| 
				+ C \| f-g\|_{L^p(\omega^n)}^\frac{1}{n},
\]
where C depends only on $\Omega$, $\omega$ and $p$.
\end{theorem}

\begin{proof} Suppose that $\Omega \subset B(0,R)=: B_R$ (the ball with the origin
at $0$  and  radius $R>0$). We write $\omega^n = h \beta^n$ in $B_R$, where 
$0< h \in C^\infty(\bar B_R)$ and we extend $f, g $ onto $B_R$ by setting $f=g=0$ 
on  $B_R \setminus \Omega$.
Therefore, $fh, gh \in L^p(B_R, \beta^n)$. 
By  \cite{ko96}, there is a unique 
$w \in PSH(B_R)\cap C(\bar B_R)$ solving
$(dd^c w)^n = |fh - gh | \beta^n$ with $w = 0$ on $\partial B_R$. The 
stability estimate for the complex Monge-Amp\`ere equation proven in \cite{ko96} says that
\[
	\| w \|_{L^\infty(\bar B_R)} \leq 
	C_1 \| fh -gh \|_{L^p(B_R, \beta^n)}^\frac{1}{n},
\]
where $C_1$ depends only on $\Omega$, $p$.
Since
\[
	\left(\omega + dd^c (u + w) \right)^n \geq \omega_u^n + (dd^c w)^n
	= f \omega^n + |fh - gh| \beta^n \geq g \omega^n,
\]
and $w \leq 0$ in $\Omega$, we can apply
the domination principle  for $\varphi: = u + w$ and 
$\psi := v + \sup_{\partial \Omega} |u-v|$ to get that
$u + w\leq v + \sup_{\partial \Omega} |u-v|$ in $\Omega$. Hence, 
\[
	w -  \sup_{\partial \Omega} |u-v| \leq v -u.
\] 
Similarly, we obtain $ v -u \leq -w + \sup_{\partial \Omega} |u-v|$.
So
\begin{align*}
	| u -v | 
	\leq \| w \|_{L^\infty} + \sup_{\partial \Omega} |u-v|
\,&	\leq \sup_{\partial \Omega} |u-v| 
					+ C_1 \|f h -g h\|_{L^p(B_R, \beta^n)} ^\frac{1}{n} \\
\,&	\leq \sup_{\partial \Omega} |u-v| + C \|f  -g \|_{L^p(\Omega, \omega^n)} ^\frac{1}{n},
\end{align*}
where $C$ depends on $\Omega, p$ and $\sup_{\bar \Omega} h$.
\end{proof}

\begin{theorem}
\label{existence-local} 
In a $C^\infty$ strictly pseudoconvex domain
there exists a unique continuous solution to the Dirichlet
problem \eqref{dirichlet-problem}. 
\end{theorem}
\begin{proof}
Suppose that $\phi_j \in C^\infty(\partial \Omega)$ converges uniformly to $\phi$ and a
 sequence of smooth functions $f_j >0$  converges to $f$ in $L^p(\omega^n)$. 
 From Theorem 1.1 in \cite{GL}, it follows that 
 for each $j$ there exists a unique smooth
 solution $u_j \in PSH(\omega)$ of the corresponding Dirichlet problem
 \begin{equation*}
 \begin{cases}
 	(\omega + dd^c v)^n & =f_j \omega^n  \quad \mbox{ in }  \Omega,\\
	 v				&= \phi_j 	\quad	\mbox{ on } \partial \Omega.
 \end{cases}
 \end{equation*}
 Hence, from Theorem~\ref{stability-local} we get that the solutions $u_j$ form a Cauchy 
 sequence in $C(\bar\Omega)$. Thus, they converge uniformly  to $u$  in 
 $PSH(\omega) \cap C(\bar \Omega)$.
 Therefore $\omega_{u_j}^n$ converge weakly to $\omega_u^n$ by Proposition~\ref{convergence}. It means that 
 $u$ is a  continuous  solution to the Dirichlet problem \eqref{dirichlet-problem}.
  Moreover, by the domination principle (Corollary~\ref{domination-principle}) this 
  solution is  unique.  The proof is completed.
\end{proof}

\section{Existence of continuous solutions on a compact Hermitian manifold}
\label{global-dp}
Let $(X,\omega)$ be a compact Hermitian manifold of complex dimension $n$. 
The constant $B>0$ in
\eqref{curvature} is used throughout this section. We denote by $C$ 
 a generic positive constant depending only on $n,B$, which may vary from line to line. We use the notation $Vol_\omega(E) :=\int_E \omega^n$ for any Borel set $E$, and write $L^p(\omega^n)$ for $L^p(X, \omega^n)$.

\subsection{$L^\infty$ a priori estimates}
\label{apriori}

We first show how the modified comparison principle coupled with  pluripotential theory techniques leads to  $L^\infty$ {\em a priori} estimates. Recall that  for a Borel set $E \subset X$
\[
	cap_\omega(E) := \sup \left\{
						\int_E (\omega + dd^c \rho)^n : \, 
						\rho \in PSH(\omega), \, 
						0 \leq \rho \leq 1
						 \right\}.
\]

\begin{prop}[\cite{DK}, Corollary 2.4]
\label{inequality-vol-cap}
There are a universal number $0< \alpha = \alpha(X,\omega)$ 
and a uniform constant $0< C = C(X, \omega)$ such that for any Borel subset 
$E \subset X$  
\[
	Vol_{\omega}(E) \leq C \exp\left( \frac{-\alpha}{cap_\omega^\frac{1}{n}(E)} \right). 
\]
Consequently, by H\"older's inequality, for any $0\leq f \in L^p(\omega^n)$, $p>1$, 
\[
	\int_E f \omega^n 
	\leq C \|f\|_{L^p(\omega^n)} 
	\exp \left( - \frac{\tilde \alpha}{ cap_\omega^\frac{1}{n}(E)}\right),
\]
where $\tilde\alpha = \alpha/q$, $1/p+ 1/q =1$.
\end{prop}


Let $h : \mathbb R_+ \rightarrow (0, \infty ) $ be an increasing function such that
\begin{equation}
\label{admissible}
	\int_1^\infty \frac{1}{x [h(x) ]^{\frac{1}{n}} }  \, dx < +\infty.
\end{equation}
In particular, $\lim_{ x \rightarrow \infty} h(x) = +\infty$. Such a function $h$ is called
{\em admissible}. If $h$ is admissible, then so is $A \, h$ for any number $A >0$.
Define
\[
	F_h(x) = \frac{x}{h(x^{-\frac{1}{n}})}.
\]
For such $F_h$ we consider the family of bounded $\omega$-psh functions 
such that their Monge-Amp\`ere measures satisfy
\begin{equation}
\label{magrowth}
	\int_E \omega_\varphi^n \leq F_h( cap_\omega (E)),
\end{equation}
for any Borel set $E \subset X$. It follows from Proposition~\ref{inequality-vol-cap}
that

\begin{corollary}
\label{lp-density} Let $\varphi\in PSH(\omega) \cap L^\infty(X)$.
If $\omega_\varphi^n = f \, \omega^n$ for $0 \leq f \in L^p(\omega^n)$, $p>1$, then
$\omega_\varphi^n$ satisfies \eqref{magrowth} for the admissible function 
$h_p(x) = C \|f\|_{L^p(\omega^n)} ^{-1}\exp(a x)$ with some universal number $a>0$.
\end{corollary}


Our next theorem is a generalization of a priori estimates in \cite{ko98}, \cite{ko03}
from the K\"ahler setting to the Hermitian one.

\begin{theorem}
\label{kappa}
Fix $ 0 < \varepsilon <1$. Let $ \varphi, \psi \in PSH (\omega)\cap L^\infty(X)$ 
be such that $\varphi \leq 0$, and $ -1 \leq \psi \leq 0$. Set 
$m(\varepsilon) = \inf_X [ \varphi - (1-\varepsilon) \psi]$, and 
$
\varepsilon_0:=
	\frac{1}{3}\min\{
	\varepsilon^n, 
	\frac{\varepsilon^3}{16 B}, 
	4 (1-\varepsilon) \varepsilon^n, 
	4 (1-\varepsilon)\frac{\varepsilon^3}{16 B} \}
$.
Suppose that 
$\omega_\varphi^n$ satisfies  \eqref{magrowth} for an admissible 
function $h$. Then, for $0<D< \varepsilon_0$,
\[
	D \leq \kappa\left[ cap_\omega ( U(\varepsilon, D))\right],
\]
where $U(\varepsilon, D ) = \{ \varphi < (1- \varepsilon) \psi + m(\varepsilon) + D \}$, 
and the function $\kappa $ is defined on the interval $(0,cap_\omega(X))$ by the formula
\[
	\kappa \left ( s^{-n} \right) =
		4\, C_n  \left \{
			\frac{1}{ \left [ h ( s )\right]^{\frac{1}{n}} } 
			+ \int_{s}^\infty  \frac{dx}{x \left[ h (x) \right]^{\frac{1}{n}}}			
					\right \},
\]
with a dimensional constant $C_n$.
\end{theorem}

The following  lemma  is the  crucial step in the proof of the theorem. 
It is an estimate of the capacity of sublevel sets. 
The proof goes through in the Hermitian setting thanks to the modified comparison 
principle (Theorem~\ref{modified-comparison-principle}). 

\begin{lemma}
\label{capacity}
Fix $0 <\varepsilon  <1 $. Let $\varphi, \psi \in PSH(\omega) \cap L^\infty(X)$ be such that
$ -1 \leq \psi \leq 0$. 
Set $m(\varepsilon) = \inf_X [ \varphi - (1-\varepsilon) \psi]$ and
\[
	U(\varepsilon, s):= \{\varphi < (1-\varepsilon) \psi + m(\varepsilon) + s \}.
\]
For any 
$0 <s, t \leq  \frac{1}{3}\min\{\varepsilon^n, \frac{\varepsilon^3}{16 B} \}$ (with $B$ defined above) one has
\[
	[(1-\varepsilon)\, t ]^n \, cap_{\omega} (U(\varepsilon, s))
	\leq 	( 1 + C) \,  \int_{ U (\varepsilon, s + 4(1-\varepsilon)\,t) }
			\omega_\varphi^n.
\]
\end{lemma}

\begin{proof}
Let $\rho \in PSH(\omega)$ be such that $0 \leq \rho \leq 1$. It follows that
\[
	U(\varepsilon, s) \subset 
	\left\{ \varphi < (1-\varepsilon) \left [ (1-t) \psi + t \, \rho
									\right ]
						+m(\varepsilon) + s\right\}.
\]
If we use the notation
\[
	m(\varepsilon, t ) := \inf_X \left( 
				\varphi - (1-\varepsilon) \left [ (1-t) \psi + t  \, \rho
									\right ]
					\right),
\]
then $ m(\varepsilon, t) \leq m(\varepsilon) \leq m(\varepsilon, t) + 2 (1-\varepsilon) \, t$.  Hence, 
\begin{align*}
	U(\varepsilon, s) &\subset V:=
	\left\{ \varphi < (1-\varepsilon) \left [ (1-t) \psi + t  \, \rho
									\right ]
						+m(\varepsilon, t) + s + 2 (1- \varepsilon)\, t \right\} \\
	&\subset U(\varepsilon, s + 4(1-\varepsilon)\, t).
\end{align*}
Then, Theorem~\ref{modified-comparison-principle} gives
\begin{align*}
	[(1-\varepsilon)t ]^n \, \int_{U(\varepsilon, s)}  (\omega + dd^c \rho)^n 
\,&	\leq \int_V \left ( \omega 
	+ (1-\varepsilon) \, dd^c \left [ (1-t) \psi + t \, \rho
						\right]  \right)^n \\
\,&	\leq \left (1+ \frac{s+ 2(1-\varepsilon ) \, t}{\varepsilon^n} \, C \right) 
		\int_V \omega_\varphi^n \\
\,&	\leq  (1+C)\int_{ U(\varepsilon, s + 4 (1- \varepsilon)\, t)} \omega_\varphi^n.
\end{align*}
Thus the lemma follows. 
\end{proof}

After rescaling $t$ the statement of  Lemma~\ref{capacity} may be rephrased
\begin{remark}
\label{caplevelset}
For any 
$0 <s \leq \frac{1}{3}\min\{\varepsilon^n, \frac{\varepsilon^3}{16 B} \}$, 
$0< t \leq \frac{4}{3} (1-\varepsilon) \min\{\varepsilon^n, \frac{\varepsilon^3}{16 B} \}$ we have
\[
	t^n \, cap_{\omega} (U(\varepsilon, s))
	\leq 	4^n C \,  \int_{ U (\varepsilon, s + t) }	\omega_\varphi^n,
\]
where $C$ is a dimensional constant.
\end{remark}

\begin{proof}[The proof of Theorem~\ref{kappa}]
For $0 < s < \varepsilon_0$, define 
\[
	a(s):= \left [ cap_\omega( U(\varepsilon, s)) \right]^\frac{1}{n} >0,
\]
and 
\[
	g(x) = [ h(x) ]^\frac{1}{n}.
\]
From Remark~\ref{caplevelset} and the property \eqref{magrowth} we infer that
for any $0 < s, t < \varepsilon_0$
\[
	t \, a(s) \leq C \, \frac{a(s+t)}{ g \left( \frac{1}{a(s+t)} \right)},
\]
where $C = 4\, (1+ C_n)^{1/n}$. 
We may assume $C = 1$ after muliplying $g$ by an appropriate constant. Hence,
\begin{equation}
\label{kappa-1}
	t 
	\leq  \frac{a(s+t)}{ a (s) \, g \left( \frac{1}{a(s+t)} \right)} .
\end{equation}
Let $0<D<\varepsilon_0$. Applying \eqref{kappa-1} for $t:= D-s$, and $0<s<D$ we obtain
\[
	D -s \leq \frac{a(D)}{a(s) g(\frac{1}{a(D)})} .
\] 
Set 
\[
	s_0 := \sup\{ 0< s < D: a(D) >  e \, a(s) \}.  
\]
Since $\lim_{t\rightarrow s^-} a (t) = a (s)$, so $s_0 < D$. It is clear that
$a(D) \leq e\, a(s_0^+) $, where $a(s^+) = \lim_{t\rightarrow s^+} a(t)$. It follows that 
\[
	D - s_0 \leq \lim_{s\rightarrow s_0^+} \frac{a(D)}{a(s) g(\frac{1}{a(D)})}
		     =	\frac{a(D)}{a(s_0^+) g(\frac{1}{a(D)})}
		     \leq \frac{e}{g(\frac{1}{a(D)})} . 
\]
Thus, the theorem will follow if we have the estimate of $s_0$ from above. We define 
by induction a strictly decreasing sequence which begins with $s_0$, and for $j \geq 0$, satisfies
\[
	s_{j+1} := \sup \{ 0< s < s_j : a( s_{j}) > e \, a(s)\}.
\]
It follows that 
\[
	a(s_j) \leq e \, a(s_{j+1}^+).
\]
By monotonicity of $a(t)$ and the definition of $s_{j+1}$ there exists 
$ s_{j+2} < t < s_{j+1}$ such that
\[
	e \, a(s_{j+2}^+) \leq e \, a(t) < a(s_{j}).
\]
Hence, we have 
\begin{equation}
\label{kappa-2}
		\frac{1}{e} a(s_{j+1}) \leq a(s_{j+2}^+) \leq \frac{1}{e} a(s_{j}) .
\end{equation}
We are ready to estimate $s_0$. Applying \eqref{kappa-1} for $s = s_{j+1}$, 
$t= s_{j} - s_{j+1}$ we have
\[
	s_{j} - s_{j+1} 
	\leq \lim_{x\rightarrow s_{j+1}^+} \frac{a (s_j)}{ a(x) g \left( \frac{1}{a(s_j)}\right)}
	=	\frac{a (s_j)}{ a(s_{j+1}^+) g \left( \frac{1}{a(s_j)}\right)}.
\]
Then, using the first inequality in \eqref{kappa-2}, we get 
\[
	s_j - s_{j+1} \leq \frac{e}{g \left( \frac{1}{a(s_j)}\right)}.
\]
Setting 
\[
	x_j := \frac{1}{a(s_j)},
\]
we have, using the second inequality in \eqref{kappa-2} and 
$a(s_{j+2}) \leq a(s_{j+2}^+)$,
\[
	\frac{x_{j+2} - x_j}{x_{j+2} g(x_{j+2})} 
		= \frac{a(s_j) - a(s_{j+2})}{a(s_{j})} \;
		\frac{1}{ g\left( \frac{1}{a(s_{j+2})} \right)}
		\geq  \frac{e-1}{ e }      \frac{1} {g \left( \frac{1}{a(s_{j+2}) }\right)}.
\]
Combining the last two estimates we have 
\[
	s_{j+2} - s_{j+3} 
	\leq \frac{e^2}{e-1} \; \frac{x_{j+2} - x_j}{x_{j+2} \,  g(x_{j+2})}
	\leq \frac{e^2}{e-1} \; \int_{x_j}^{x_{j+2}} \frac{dx}{x \, g(x) },
\]
as $x  g(x)$ is an increasing function.
Therefore,
\[
	s_2 = \sum_{j=0}^\infty (s_{j+2} - s_{j+3}) 
		\leq \frac{e^2}{e-1} \;
			 \sum_{j=0}^\infty \int_{x_j}^{x_{j+2}} \frac{dx}{x \, g(x) }.
\]
Since
\[
	 \sum_{j=0}^\infty \int_{x_j}^{x_{j+2}} \frac{dx}{x \, g(x) } 
	 \leq 2 \int_{x_0}^\infty 	\frac{dx}{x \, g(x) } 
	 \leq 2  \int_{\frac{1}{a(D)}}^\infty 	\frac{dx}{x \, g(x) }, 
\]
and
\[
	s_0 = (s_0- s_1) + (s_1-s_2) +s_2 
		\leq \frac{2 \, e}{g\left( \frac{1}{a(D)}\right)} + s_2,
\]
we have
\[
	s_0 \leq \frac{2 \, e}{g\left( \frac{1}{a(D)}\right)}
		+  \frac{2 e^2}{e-1}  \int_{\frac{1}{a(D)}}^\infty 	\frac{dx}{x \, g(x) }.
\]
Consequently
\[
	D  \leq 
				\frac{3 e}{g \left( \frac{1}{a(D)} \right) } 
				+\frac{2 e^2}{e-1} \int_{\frac{1}{a(D)}}^\infty  \frac{dx}{x \, g(x)} . 
\]
This is equivalent to the statement of the
 theorem.
\end{proof}

In the next step, we will see how Theorem~\ref{kappa} implies a   
$L^\infty$ {\em a priori} bound on the solutions of the Monge-Amp\`ere equation
with the right hand side in $L^p$, $p>1$. 

Suppose that $\varphi\in PSH(\omega)\cap L^\infty(X)$, $\sup_{X} \varphi = 0$
satisfies
\begin{equation}
\label{maa}
	\omega_\varphi^n = f \omega^n,
\end{equation}
where $0\leq f \in L^p(\omega^n)$, $p>1$. Then, $\omega_\varphi^n$ satisfies 
\eqref{magrowth} for $h(x)= C \|f\|_{L^p(\omega^n)}^{-1} \exp(ax)$, $a>0$
(Corollary~\ref{lp-density}).
Let $\hbar$ be the inverse function of $\kappa$ in Theorem~\ref{kappa}.
Then, $\hbar$ is also an increasing function. 

\begin{corollary}
\label{fomula-Linfty-bound}
Let $\varphi, f$ be as in \eqref{maa}. 
There exists a constant $0 < H = H(h)$, depending only on $h$, $X$, 
and $\omega$ such that
\begin{equation}
\label{L^infty}
	- H \leq \varphi \leq 0.
\end{equation}
Moreover, we have  for $b \geq 1$,
\begin{equation}
\label{changeL^infty}
	H (b^{-n} h) \leq b H (h).
\end{equation}
\end{corollary}

\begin{proof}
Applying Theorem~\ref{kappa} for $\psi=0$, and $\varepsilon = 1/2$,
 we have
\begin{equation}
\label{hbar-cap}
	s \leq \kappa \left [ cap_\omega \left( \{\varphi < \inf_X \varphi + s \}\right)\right]
	\Rightarrow 
	\hbar( s ) \leq cap_\omega \left( \{\varphi < \inf_X \varphi + s \}\right)
\end{equation}
for $0< s<\varepsilon_0$. Moreover, Proposition 2.5 in \cite{DK} says that
\[
	cap_\omega\left( \{\varphi < \inf_X \varphi + s \}\right)
	\leq \frac{C ||| \varphi |||_{L^1(\omega^n)}}{ | \inf_X \varphi +s| } ,
\]
where $C$ and
\[
	||| \varphi |||_{L^1(\omega^n)}
		:= \sup \left \{ \int_X |\varphi| \omega^n :
				 \varphi \in PSH(\omega), \, \sup_X \varphi =0 \right\}
\]
are uniform constants.
Two last inequalities imply
\[
	\hbar(s) \leq \frac{C}{ | \inf_X \varphi +s| } ||| \varphi |||_{L^1(\omega^n)}.
\]
Therefore, 
\[
	| \inf_X \varphi | \leq s + \frac{C ||| \varphi |||_{L^1(\omega^n)}}{\hbar(s)} 
\]
for $0 < s < \varepsilon_0$.  This gives \eqref{L^infty}. 
In order to obtain \eqref{changeL^infty}, we proceed as follows.
Let $\phi \in PSH(\omega) \cap L^\infty(X)$, $\sup_X \phi =0$, be such that 
for any Borel set $E$
\[
	\int_E \omega_\phi^n \leq b^n F_h(cap_\omega(E)).
\]
It follows from the formula for the function $\kappa$  in  Theorem~\ref{kappa} that the function $\kappa'$ for $b^{-n} h$ is  $ b \kappa$. The above argument implies that
\[
	|\inf_X \phi | \leq s + \frac{C ||| \varphi |||_{L^1(\omega^n)}}{\hbar(\frac{s}{b})}, 
\]
 where we used the fact that the inverse of
$
	\kappa' ( . )= b \kappa ( .)
$
is
$
	\hbar'( . ) = \hbar ( \frac{1}{b} \; .).
$
From the formula for the
function $\kappa$ associated to the admissible function $h(x)= C \exp(ax)$, $a>0$,  it follows that for $b\geq 1$, $0< x < \varepsilon_0$, 
\[
	b \hbar(\frac{x}{b})\geq \hbar(x).
\]
Thus, for $0<s<\varepsilon_0$,
\[
	|\inf_X \phi | \leq 
	b \left(\frac{s}{b} 
	+ \frac{C ||| \varphi |||_{L^1(\omega^n)}}{b\hbar(\frac{s}{b})} \right)
	\leq b \left(s
	+ \frac{C ||| \varphi |||_{L^1(\omega^n)}}{\hbar(s)} \right).
\]
The corollary follows.
\end{proof}

\subsection{Weak solutions to the complex Monge-Amp\`ere equation}
\label{weak-solution}

In this section we are going to study the existence of weak solutions for the Monge-Amp\`ere equation on $X$.
Let $0 \leq f \in L^p(\omega^n )$, $p>1$. We wish to solve the  equation
\begin{equation}
\label{ma}
\begin{cases}
	u \in PSH(X , \omega) \cap L^\infty(X), \\
	(\omega + dd^c u)^n = f \, \omega^n.
\end{cases}
\end{equation}
In general $\omega$ is not closed, and then  the appropriate statement of \eqref{ma} is that
 there exist a constant 
$c>0$ and a bounded (or continuous) $\omega$-psh function $u$ such that 
\begin{equation}
\label{maeh}
	(\omega + dd^c u)^n = c \, f \, \omega^n.
\end{equation}

\begin{remark}
\label{no-bounded-solution}
If $f \equiv 0$, then the equation \eqref{maeh} has no bounded solution.
\end{remark}

\begin{proof}
It is a immediate consequence of the inequality \eqref{lowermass} below
as an open subset has a positive capacity.
\end{proof}

\begin{theorem}
\label{existence}
Let $0\leq f \in L^p(\omega^n)$, $p>1$, be such that $\int_X f \omega^n >0$. 
There exist a constant $ c>0$ and $u\in PSH(\omega)\cap C(X)$ satisfying the
equation \eqref{maeh}.
\end{theorem}

\begin{proof} 
Choose $f_j \in L^p(\omega ^n )$ smooth, strictly positive and  converging  to  $f$ 
in $L^p(\omega^n)$. By a theorem of Tosatti and Weinkove \cite{TW2}, for each 
$j \geq 1$, there exist a unique $u_j \in PSH( \omega) \cap C^\infty(X)$ with 
$\sup_X u_j =0$ and a unique constant $c_j>0$ such that 
\begin{equation}
\label{as}
	(\omega + dd^c u_j)^n = c_j \, f_j \, \omega^n.
\end{equation}

\begin{lemma}
\label{uniform-L^p}
The sequence $\{c_j\}$ is bounded away from $0$ and bounded from above. 
In particular, the family $\{c_j f_j\}$ is bounded in $L^p(\omega^n)$.
\end{lemma}

\begin{proof}
We first show that $c_j$'s are uniformly bounded from above. 
Since $f_j \rightarrow f$ in $L^1(\omega^n)$, we also have 
$f_j^\frac{1}{n} \rightarrow f^\frac{1}{n}$
in $L^1(\omega^n)$. Because $\int_X f \omega^n>0$, $\int_X f^\frac{1}{n} \omega^n >0$
one obtains
\[
	\int_X f_j^\frac{1}{n} \omega^n 
	> \frac{\int_X f^\frac{1}{n} \omega^n}{2}
	>0
\] 
for   $ j > j_0$ ($j_0 \geq 1$ depends on $f$). 
The pointwise arithmetic-geometric means inequality  implies that
\[
	(\omega + dd^c u_j) \wedge \omega^{n-1}
	\geq \left[\frac{(\omega + dd^c u_j)^n}{\omega^n}\right]^\frac{1}{n} \omega^n 
	= (c_j \, f_j)^\frac{1}{n} \omega^n. 
\]
Hence,
\[
	c_j^\frac{1}{n} \int_X f_j^\frac{1}{n} \omega^n 
		\leq \int_X (\omega +dd^c u_j) \wedge \omega^{n-1}.
\]
It follows that for $j>j_0$,
\begin{equation}
\label{L1b}
	c_j^\frac{1}{n} 
	\leq  \frac{2}{\int_X f^\frac{1}{n} \omega^n} \int_X (\omega +dd^c u_j) \wedge \omega^{n-1}.
\end{equation}
To end the proof we need to show that the right hand side is uniformly bounded
from above. 
Since  $\sup_X u_j = 0$, it follows  that
\begin{equation}
	\int_X |u_j| \omega^n \leq C_1,  
\end{equation}
with a uniform constant $C_1$ (see e.g.  \cite{DK}, Proposition 2.5). 
Hence, using the Stokes theorem we have
\begin{align*}
	\int_X dd^c u_j \wedge \omega^{n-1} 
&	= \int_X u_j \wedge dd^c(\omega^{n-1}) \\
&	\leq B \int_X |u_j| \omega^n \\
&	\leq B \, C_1. 
\end{align*}
 Combining this with \eqref{L1b} we conclude that
$\{c_j\}$ is bounded from above.

It remains to verify that $\{c_j\}$ is bounded away from $0$. 
Applying Remark~\ref{caplevelset} for $\varepsilon = 1/2$, $\psi =0$ and 
$0 \geq \varphi \in PSH(\omega) \cap L^\infty(X)$, with 
 $S= \inf_X \varphi $, we get 
for $0<s,t < \varepsilon_0$,
\begin{equation}
\label{lowermass}
	t^n \, cap_\omega( \varphi < S + s) 
		\leq C \int_{\{\varphi < S+ s + t\}}  \omega_\varphi^n.
\end{equation}
From  Remark~\ref{caplevelset} for $\varphi:= u_j$ with $\inf_X u_j = S_j$ and the H\"older inequality, it follows that, 
for $0 < s,t < \varepsilon_0$,
\begin{align*}
	t^n cap_\omega (u_j < S_j + s ) 
\, & 	\leq C \int_{\{ u_j < S_j + s +t\}} c_j \, f_j  \omega^n\\
\, & 	\leq C \, c_j \|f_j \|_{L^p(\omega^n)} [Vol_\omega( \{ u_j < S_j + s + t\})]^\frac{1}{q},
\end{align*} 
where $1/p+1/q =1$. Therefore, for fixed  $0< s= t < \varepsilon_0 ,$
\[
	cap_\omega (u_j < S_j + s ) 
	 \leq s^{-n}  C\, c_j \|f _j \|_{L^p(\omega^n)} [Vol_\omega( \{ u_j < S_j + 2 s\})]^\frac{1}{q}
	 := c_j\,  C_1  s^{-n} ,
\]
with $C_1$ depending also on $X$ and $f$.
From Theorem~\ref{kappa} we know that 
\[
	s \leq \kappa(cap_\omega(\{ u_j < S_j + s \}))
	\leq \kappa( c_j \, C_1  s^{-n}). 
\]
Since $\lim_{x\rightarrow 0^+} \kappa(x) = 0$, 
$c_j$ must be uniformly bounded  away from $0$. 
\end{proof}

We proceed to finish the proof of the theorem.

\medskip
{\em Uniform bound of  $||u_j ||_{L^{\infty } }$ }.  Since $c_j\, f_j$ are uniformly
bounded in $L^p(\omega ^n )$, the $L^\infty$ a priori estimate from \cite{DK} 
(or Corollary~\ref{fomula-Linfty-bound})
 implies that $\{u_j\}$ are uniformly bounded. Thus there exists $H >0$ such that 
$-H \leq  u_j \leq 0$ for every $j$.  By rescaling 
we may assume from now on that $H = 1$. Moreover, by passing to a subsequence,
it is  assumed that $\{u_j\}_{j=1}^\infty$ is a Cauchy sequence  in $L^1(\omega ^n  )$ and $0<c = \lim c_j .$

\medskip
{\em The existence of a  continuous solution.} 
Let us  use the notation 
\[
	S_{kj}:=  \inf_X (u_k - u_j) \leq 0,\quad \quad 
	M_{kj}:=  \sup_X (u_k - u_j) \geq 0.
\]
We are going to show that both $S_{kj} \rightarrow 0$ and $M_{kj} \rightarrow 0$ 
as $k,j \rightarrow +\infty$,
 arguing by contradiction. Suppose that $S_{kj}$ does not converge to $0$ as $k, j \rightarrow  +\infty$.
So there exists $0 < \varepsilon <1$
 such that 
\[
	S_{kj} \leq - 4 \, \varepsilon,
\]
for arbitrarily large  $k\neq j$. In order to simplify the notation, we write
$\varphi: = u_k$, $\psi := u_j$, and $S:= S_{kj}$, for such a pair $j,k$.  Put 
\[
	m(\varepsilon):= \inf_X \left[ \varphi - (1 -\varepsilon) \psi \right].
\]
It follows that
\[
	m(\varepsilon) \leq  S . 
\]
Applying Remark~\ref{caplevelset} we have, for any 
$0 <s, t < \varepsilon_0$ (see the definition in Theorem~\ref{kappa}),
\begin{equation}
\label{end-1}
	t^n \, cap_{\omega} (U(\varepsilon, s))
	\leq 	C 
		\,  \int_{ U (\varepsilon, s + t) }	\omega_\varphi^n,
\end{equation}
where
\[
	U(\varepsilon, s):= \{\varphi < (1-\varepsilon) \psi + m(\varepsilon) + s \}.
\]
Since $-1 \leq \varphi, \psi \leq 0$, $S < - 4\varepsilon$, $0<s , t < \varepsilon _0$, we have the following  inclusions
\[
	U(\varepsilon, s+t) \subset \{ \varphi < \psi + S + \varepsilon +s + t \}
	\subset \{ \varphi < \psi - \varepsilon \}
	\subset \{ |\varphi - \psi | > \varepsilon\}.
\]
Thus, from \eqref{end-1} and H\"older's
inequality, with $1/p + 1/q =1$, we get that
\begin{align*}
	t^n \, cap_{\omega} (U(\varepsilon, s))
\,&	\leq 	C \,  \int_{ \{ | \varphi - \psi | > \varepsilon\} }	\omega_\varphi^n \\
\,&	\leq C \|c_k f_k\|_{L^p(\omega^n)} \, [Vol_\omega (\{| \varphi - \psi | > \varepsilon\})]^\frac{1}{q}
\end{align*}
(since $\omega_\varphi^n = c_k f_k \omega^n$).  We have already  seen that $c_k \, f_k$ is uniformly
bounded in $L^p (\omega ^n )$.
Hence, for fixed $0< s=t = D < \varepsilon_0$,
\begin{align*}
	 cap_{\omega} (U(\varepsilon, D)) 
\,&	 \leq D^{-n} C(n) \| c_k \, f _k\|_{L^p(\omega^n)} \, [Vol_\omega (\{| \varphi - \psi | > \varepsilon\})]^\frac{1}{q} \\
\,&	 \leq C_2 \, [Vol_\omega (\{| \varphi - \psi | > \varepsilon\})]^\frac{1}{q},
\end{align*}
where $C_2$ is a constant independent of $j , k$. Next, we  apply Theorem~\ref{kappa}, after  taking  values of  $\kappa$ of both sides
of the above inequality  
\[
	D \leq \kappa \left[ cap_{\omega} (U(\varepsilon, D)) \right]
	\leq \kappa \left[ C_2 \, [Vol_\omega (\{| \varphi - \psi | > \varepsilon\})]^\frac{1}{q} \right]. 
\]
This leads to a contradiction because 
$\lim_{x\rightarrow 0^+}\kappa(x) =0 $, and
\[
	Vol_\omega (\{| \varphi - \psi | > \varepsilon\}) 
	= Vol_\omega (\{ |u_k - u_j | > \varepsilon \} ) \rightarrow 0 \quad \quad
	\mbox{ as } \quad k, j \rightarrow +\infty.
\]
Thus, $S_{kj} \rightarrow 0$ as $k, j \rightarrow +\infty$.  Also  $M_{kj} \rightarrow 0$ as $k, j \rightarrow +\infty$ since $M_{kj} = - S_{jk}$.
Hence, 
\[
	|u_k - u_j | \leq |S_{kj}| + |M_{kj}| \rightarrow 0
	\quad \mbox{ as } \quad
	k,j \rightarrow +\infty.
\]
We conclude that $\{ u_j\}_{j=1}^\infty$ is a Cauchy sequence in $PSH(\omega) \cap C(X)$.  Let $u$ and $c$ be the  limit points of $\{u_j\}$ and $\{c_j\}$ respectively. Then
the continuous function $u \in PSH(\omega) \cap C(X)$  solves
\[
	(\omega + dd^c u)^n = c \, f \, \omega^n,
\]
in the weak sense of currents.
\end{proof}

It is worth to record here that from the above argument we get a weak  stability statement. 
\begin{corollary}
\label{Linfty-L1}
Let $\{u_j\}_{j=1}^\infty \subset PSH(\omega) \cap C(X)$ be such that
$\sup_X u_j =0$. 
Suppose that for every $j \geq 1$,
\[
	\omega_{u_j}^n = f_j \omega^n,
\]
where $f_j$'s are uniformly bounded in $L^p (\omega^n)$, $p>1$. If $\{ u_j \}$ is 
Cauchy in $L^1(\omega^n)$, then it is Cauchy in $PSH (\omega) \cap C(X)$.
\end{corollary}

\bigskip



\end{document}